\documentclass[12pt, reqno]{amsart}
\usepackage[active]{srcltx}
\usepackage{latexsym}
\usepackage{amssymb}
\usepackage{amsmath}
\usepackage{mathrsfs}

\usepackage{fullpage}
\usepackage{setspace}
\onehalfspace

\allowdisplaybreaks

\usepackage{amsthm}
\theoremstyle{plain}
\newtheorem{X}{X}[section]
\newtheorem{Lem}[X]{Lemma}
\newtheorem{Thm}[X]{Theorem}
\newtheorem{Prop}[X]{Proposition}

\usepackage{hyperref}
\hypersetup{
    pdftoolbar=true,                        
    pdfmenubar=false,                        
    pdffitwindow=false,                      
    pdfstartview={FitH},                    
    pdftitle={Strings of congruent primes in short intervals},    
    pdfauthor={Tristan Freiberg},    
    pdfcreator={Tristan Freiberg},
    pdfproducer={Tristan Freiberg},
    pdfnewwindow=true,                       
    colorlinks=true,
    linkcolor={black},                      
    citecolor={black},
    filecolor={black},
    urlcolor={black},         
}

\numberwithin{equation}{section}

\newcommand{\br}[1]{\ensuremath{\left(#1\right)}} 
\newcommand{\Br}[1]{\ensuremath{\left\{#1\right\}}}   

\newcommand{\ab}[1]{\ensuremath{\vert#1\vert}} 
\newcommand{\abs}[1]{\ensuremath{\left\lvert#1\right\rvert}}

\newcommand{\sums}[2]{\ensuremath{\sum_{\substack{#1 \\ #2}}}}
\newcommand{\sumss}[3]{\ensuremath{\sum_{\substack{#1 \\ #2 \\ #3}}}}

\newcommand{\prods}[2]{\ensuremath{\prod_{\substack{#1 \\ #2}}}}

\newcommand{\Sum}{\sideset{}{^{\prime}}{\sum}}
\newcommand{\Sums}[2]{\sideset{}{^{\prime}}{\sum}_{\substack{#1 \\ #2}}}

\newcommand{\SumS}{\sideset{}{^{*}}{\sum}}

\newcommand{\h}{\ensuremath{\mathcal{H}}} 
\newcommand{\hd}{\ensuremath{\h^{+}}}

\renewcommand{\le}{\leqslant}
\renewcommand{\ge}{\geqslant}
 
\renewcommand{\O}[2]{\ensuremath{\Omega(#1;#2)}} 
\renewcommand{\v}{\ensuremath{\abs{\Omega(p)}}} 
\newcommand{\vd}[1]{\ensuremath{\abs{\Omega(#1)}}} 
\newcommand{\vs}[1]{\ensuremath{\abs{\Omega^{*}(#1)}}} 
\newcommand{\od}[1]{\ensuremath{\Omega(#1)}}
\newcommand{\os}[1]{\ensuremath{\Omega^{*}(#1)}}
\newcommand{\op}[1]{\ensuremath{\Omega^{+}(#1)}}

\title{Strings of congruent primes in short intervals}
\author{Tristan Freiberg}

\begin{document}

\begin{abstract}
Fix $\epsilon > 0$, and let $p_1 = 2, p_2 = 3,\ldots$ be the sequence of all primes. We prove that if $(q,a) = 1$ then there are infinitely many pairs $p_r,p_{r+1}$ such that $p_r \equiv p_{r+1} \equiv a \bmod q$ and $p_{r+1} - p_r < \epsilon \log p_r$. The proof combines the ideas of Shiu \cite{S2000} and Goldston-Pintz-Y{\i}ld{\i}r{\i}m \cite{GPY2009}.
\end{abstract}

\maketitle

\section{Introduction}\label{Section 1}

Fix any $\epsilon > 0$. In 2005, Goldston, Pintz and Y{\i}ld{\i}r{\i}m proved \cite{GPY2005, GPY2009} that there are arbitrarily large $x$ for which there are at least two primes in the interval $(x,x + \epsilon \log x]$, thus establishing the longstanding conjecture that there are infinitely many pairs of consecutive primes $p_r, p_{r+1}$ with $p_{r+1}-p_r < \epsilon \log p_r$. 

In \cite{GPY2006} they extended their original argument to prove that there are arbitrarily large $x$ for which there are at least two
primes in the interval $(x, x + \epsilon \log x]$ which are both in the arithmetic progression $a \bmod q$, provided $(q,a) = 1$. However one cannot deduce that these are consecutive primes for there might be a prime in-between them that is not $\equiv a \bmod q$. Hence one can only deduce that \emph{either} there are infinitely many pairs of consecutive primes $p_r \equiv p_{r+1} \equiv a \bmod q$ with $p_{r+1}-p_r < \epsilon \log p_r$, \emph{or} that there are infinitely many triples of consecutive primes $p_r,p_{r+1},p_{r+2}$ with $p_{r+2}-p_r < \epsilon \log p_r$. Presumably both statements are true but one can only deduce that one of them is true, and one does not know which one, from the result in \cite{GPY2006}.

In \cite{S2000}, Shiu proved an old conjecture of Chowla that there are infinitely many pairs of consecutive primes $p_r, p_{r+1}$ which are
both $\equiv a \bmod q$. Indeed he was even able to extend this to $k$ consecutive primes. In this paper we will combine the methods of Goldston-Pintz-Y{\i}ld{\i}r{\i}m and of Shiu to establish the following hybrid of those results:

\begin{Thm}\label{T1.1}
Let $q \ge 3$ and $a$ be integers with $(q,a) = 1$, and fix any $\epsilon > 0$. There exist infinitely many pairs of consecutive primes $p_r,p_{r+1}$ such that $p_r \equiv p_{r+1} \equiv a \bmod q$ and $p_{r+1} - p_r < \epsilon\log p_r$.
\end{Thm}

\section{Preliminaries}\label{Section 2}

In this section we will state two key technical propositions, to be proved in sections \ref{Section 4} and \ref{Section 5}. The first proposition requires some preparation. We begin by quoting the Landau-Page theorem, a proof of which can be found in \cite[Chapter 14]{D2000}. This theorem is used to handle problems arising from possible irregularities in the distribution of primes, hence in Bombieri-Vinogradov type theorems (see Lemma \ref{L4.2}), caused by potential Siegel zeros.
\begin{Lem}[Landau-Page theorem]\label{L2.1}
There exists a constant $c$ such that the following holds for any $Y > c$. There is at most one integer $q_0 \le Y$, and at most one real primitive character $\chi_0 \bmod q_0$, such that
\[
L(1 - \delta,\chi_0,q_0) = 0 \quad \textrm{for some} \quad \delta \le \frac{1}{3\log Y}.
\]
If $q_0$ exists, then $q_0 > (\log Y)^2$. We call $\chi_0$ an exceptional character and $q_0$ an exceptional modulus.
\end{Lem}
Throughout, we fix a number $\epsilon > 0$, we let $H$ be a real parameter tending monotonically to infinity, and we set $N := \exp(H/\epsilon)$, that is $H = \epsilon\log N$. If there is an exceptional modulus $q_0 := q_0(H) \le \exp(H/\epsilon(\log(H/\epsilon))^2) = N^{1/(\log\log N)^2}$, let $p_0 := p_0(H)$ be its greatest prime factor; otherwise let $p_0 = 1$.
 
For all sufficiently large $H$, either
\begin{align}\label{2.1}
\textrm{$p_0 = 1$ or $p_0$ is a prime with $p_0 > \log H$.}
\end{align}
To see this, note that all real primitive characters are products of Legendre symbols with different odd primes, and possibly either the unique real character $\bmod$ 4 or one of the two primitive real characters $\bmod$ 8. Thus if $q_0$ exists it is of the form $2^{\alpha}p_1\cdots p_k$, where $\alpha \le 3$ and the $p_i$'s are distinct odd primes. If this is the case and $p_0 \le \log H$, then the prime number theorem implies $q_0 \ll \exp((1+o(1))\log H) \ll \log N$, but Lemma \ref{L2.1} states that  $q_0 > (\log N/(\log\log N)^2)^2$.

We let $Q := Q(H)$ be a positive integer, upon which we will impose the following conditions:
\begin{align}
& \textrm{$Q$ is composed only of primes $p \le H$,} \label{2.2} \\
& \textrm{$Q$ is divisible by all primes $p \le \log H$,} \label{2.3} \\
& \textrm{$Q \le \exp\br{cH/(\log H)^2}$ for some constant $c > 0$,} \label{2.4} \\
& \textrm{if $p_0(H) \ne 1$ then $p_0(H)$ does not divide $Q$.} \label{2.5}
\end{align}
We let
\begin{align}\label{2.6}
\h := \{Qx+h_1,\ldots,Qx+h_{k}\}, \quad h_1,\ldots,h_{k} \in [1,H] \cap \mathbb{Z},
\end{align}
denote a set of distinct linear forms, and we define
\begin{align}\label{2.7}
\Lambda_R(n;\h,j) := \frac{1}{j!} \Sums{d \mid P(n;\h)}{d \le R} \mu(d)(\log R/d)^j, 
\end{align}
where $\sum'$ denotes summation over indices coprime with $Qp_0$, and
\begin{align}\label{2.8}
P(n;\h) := (Qn+h_1)\cdots (Qn+h_{k}).
\end{align}
Finally, we let
\begin{align*}
\vartheta(n) := 
\begin{cases}
\log n & \textrm{if $n$ is prime,} \\
0      & \textrm{otherwise.}
\end{cases}
\end{align*}

\begin{Prop}\label{P2.2}
Given $\epsilon > 0$ and sufficiently large $H$, let $N$ and $p_0 = p_0(H)$ be as defined earlier, and let $Q = Q(H)$ be a positive integer satisfying \eqref{2.2} -- \eqref{2.5}. Fix positive integers $k$ and $\ell$, and let $\h = \{Qx+h_1,\ldots,Qx+h_k\}$ be a set of distinct linear forms with $h_1,\ldots,h_k \in [1,H] \cap \mathbb{Z}$ and $(Q,h_1,\ldots,h_k) = 1$. Let $h \in [1,H] \cap \mathbb{Z}$ and suppose $(Q,h) = 1$, and let $R = N^{1/4 - \epsilon'}$ for some $\epsilon' \in (0,1/4)$. As $H \to \infty$, we have
\begin{align}\label{2.9}
\frac{1}{N}\br{\frac{\phi(Q)}{Q}}^{k} \sum_{N < n \le 2N} \Lambda_R(n;\h,k+\ell)^2 \sim 
\binom{2\ell}{\ell} \frac{(\log R)^{k+2\ell}}{(k+2\ell)!}
\end{align}
and
\begin{align}\label{2.10}
& \frac{1}{N} \br{\frac{\phi(Q)}{Q}}^{k} \sum_{N < n \le 2N} \vartheta(Qn+h) \Lambda_R(n;\h,k+\ell)^2 \nonumber \\ & \hspace{100pt} \sim
\begin{cases}
\displaystyle\frac{Q}{\phi(Q)}\binom{2\ell}{\ell} \frac{(\log R)^{k+2\ell}}{(k+2\ell)!}  & \textrm{if $Qx+h \not\in \h$,} \\ & \\
\displaystyle\binom{2(\ell+1)}{\ell + 1} \frac{(\log R)^{k+2\ell+1}}{(k+2\ell+1)!} & \textrm{if $Qx+h \in \h$.} 
\end{cases}
\end{align}
\end{Prop}

\begin{Prop}\label{P2.3}
Let $q \ge 3$ and $a$ be integers with $(q,a) = 1$, and for a given $H$, let $p_0 = p_0(H)$ be as defined earlier. There is an infinite sequence of integers $H_1 < H_2 < \ldots$ such that for any $i$, taking $H = H_i$, there exists a positive integer $Q = Q(H)$, divisible by $q$ and satisfying \eqref{2.2} -- \eqref{2.5}, such that
\begin{align}\label{2.11}
\abs{S} - \abs{T} \gg_q H\br{\frac{\phi(Q)}{Q}},
\end{align}
where
\begin{align}\label{2.12}
\begin{split}
S = S(H) & := \{h \in (0,H] : \textrm{$(Q,h) = 1$ and $h \equiv a \bmod q$}\}, \\
T = T(H) & := \{h \in (0,H] : \textrm{$(Q,h) = 1$ and $h \not\equiv a \bmod q$}\}.
\end{split}
\end{align}
The implied constant in \eqref{2.11} depends at most on $q$. 
\end{Prop}

\section{Proof of Theorem 1.1}\label{Section 3}

Fix integers $q \ge 3$ and $a$ with $(q,a) = 1$. Recall that $H = \epsilon\log N$, with $\epsilon > 0$ fixed, and $p_0$ is the greatest prime factor of the exceptional modulus $q_0 \le N^{1/(\log\log N)^2}$, if it exists, otherwise $p_0 = 1$. We choose $H$, $Q = Q(H)$, $S = S(H)$, and $T = T(H)$ as in Proposition \ref{P2.3}, so that $Q$ is divisible by $q$ and satisfies \eqref{2.2} -- \eqref{2.5}, and 
\begin{align}\label{3.1}
\frac{Q}{\phi(Q)}\frac{\abs{S} - \abs{T}}{\log N} \ge c(q)\epsilon
\end{align}
for some constant $c(q) > 0$, depending on $q$ at most.

We fix positive integers $k,\ell$ (to be specified later), and we let $\h = \{Qx+h_1,\ldots,Qx+h_k\}$ be a set of distinct linear forms such that, for each $i$, $h_i \in [1,H] \cap a \bmod q$ and $(Q,h_i) = 1$. We let $R = N^{1/4 - \epsilon'}$ with $0 < \epsilon' < 1/4$ (to be specified later), and we put
\begin{align*}
& \mathscr{L} := \\ & \frac{1}{N}\br{\frac{\phi(Q)}{Q}}^k \sum_{N < n \le 2N} 
\br{\sum_{h \in S} \vartheta(Qn+h) - \sum_{h \in T} \vartheta(Qn+h) - \log 3QN}\Lambda_{R}(n;\h,k+\ell)^2.
\end{align*}
We now show that if $\mathscr{L} > 0$ for a sequence of numbers $N$, tending to infinity, then Theorem \ref{T1.1} follows.

Let 
\begin{align*}
A_n & := \{p \in (Qn,Qn+H] : p \equiv a \bmod q\} = \{p : p = Qn+h, h \in S\} \\
B_n & := \{p \in (Qn,Qn+H] : p \not\equiv a \bmod q\} = \{p : p = Qn+h, h \in T\}.
\end{align*}
If $\mathscr{L} > 0$, then there is some $n \in (N,2N]$ such that
\begin{align*}
\abs{A_n}\log (Qn+H) \ge \sum_{h \in S} \vartheta(Qn+h) >  \sum_{h \in T} \vartheta(Qn+h) + \log 3QN \ge
\abs{B_n}\log Qn + \log 3QN.
\end{align*}
Now 
\[
\abs{A_n}\log\br{1 + H/Qn} \le \abs{A_n}H/Qn \le H^2/QN < \log(3/2)
\]
if $N$ is sufficiently large, and so
\[
\log(3/2) + \br{\abs{A_n} - \abs{B_n}}\log Qn > \log 3QN
\]
and hence, as $n \le 2N$, $\abs{A_n} - \abs{B_n} > 1$. But as these are integers, $\abs{A_n} \ge \abs{B_n} + 2$, and so, by the pigeonhole principle, $A_n$ contains a pair of consecutive primes $p_r,p_{r+1}$. These primes satisfy $p_{r+1} - p_r < H < \epsilon\log QN < \epsilon\log p_r$.

Now, by our choice of $\h$, a straightforward application of Proposition \ref{P2.2} yields
\begin{multline*}
\mathscr{L} =
\binom{2\ell}{\ell}\frac{(\log R)^{k+2\ell}}{(k+2\ell)!} \\ \times
\Bigg\{\frac{Q}{\phi(Q)}\sums{h \in S}{Qx+h \not\in \h} 1 + 
\frac{2(2\ell+1)}{\ell+1}\frac{\log R}{k+2\ell+1} \sums{h \in S}{Qx+h \in \h} 1 -
\frac{Q}{\phi(Q)}\sum_{h \in T} 1 - (1+o(1))\log 3QN\Bigg\}.
\end{multline*}
We have
\[
\sums{h \in S}{Qx+h \in \h} 1 = k, \qquad \sums{h \in S}{Qx+h \not\in \h} 1 = \abs{S} - k,
\]
$\log R = (1/4 - \epsilon')\log N$, and $\log 3QN \sim \log N$ by \eqref{2.4}, therefore
\begin{align*}
& \mathscr{L} = \binom{2\ell}{\ell}\frac{(\log R)^{k+2\ell}}{(k+2\ell)!}\log N \\ & \hspace{60pt} \times
\Br{\frac{Q}{\phi(Q)}\frac{\abs{S} - \abs{T}}{\log N} + 
\frac{2(2\ell+1)}{\ell+1}\frac{k}{k+2\ell+1}\br{\frac{1}{4} - \epsilon'} - (1 + o(1))}.
\end{align*}
We have written $o(1)$ for $kQ/(\phi(Q)\log N)$, because $Q/\phi(Q) \ll \log\log Q \ll \log\log N$. 

By choosing $\ell = [\sqrt{k}]$ and $k$ sufficiently large, the bracketed expression $\Br{\cdots}$ above is, by \eqref{3.1},
\[
\ge c(q)\epsilon + 1 - 5\epsilon' - (1+o(1)) =
c(q)\epsilon - 5\epsilon' - o(1).
\]
By choosing $\epsilon' = c(q)\epsilon/10$ (we may assume that $\epsilon$ is small enough so that $\epsilon' < 1/4$), we deduce that 
\begin{align}\label{3.3}
\mathscr{L} \gg_{k} c(q)\epsilon(\log N)^{k+2\ell+1}
\end{align}
holds if $N$ is sufficiently large. By Proposition \ref{P2.3}, we may choose $H$, equivalently $N$, from a sequence of numbers tending to infinity, and Theorem \ref{T1.1} follows.

\section{Proof of Proposition 2.2}\label{Section 4}

The estimates \eqref{2.9} and \eqref{2.10} of Proposition \ref{P2.2} are essentially the same as estimates already in the literature, so we will only outline a proof of each of them, referring to \cite{GMPY2006} and \cite{GPY2006} for details. 

Let $Q = Q(H)$ satisfy \eqref{2.2} and \eqref{2.3}. For a set of distinct linear forms $\h$, as in \eqref{2.6}, and positive integers $d$, we define
\[
\od{d} = \O{d}{\h} := \{n \bmod d : P(n;\h) \equiv 0 \bmod d\},
\]
where $P(n;\h)$ is as in \eqref{2.8}. A Chinese remainder theorem argument shows that $n \bmod d \in \od{d}$ if and only if $p^r \mid\mid P(n;\h)$ for every $p^r \mid\mid d$, and so $\abs{\od{d}}$ defines a multiplicative function of $d$. Thus, if we define
\begin{align}\label{4.1}
\begin{split}
\lambda_R(d;j) := 
\begin{cases}
\frac{1}{j!}\mu(d)(\log R/d)^j & \textrm{if $d \le R$,} \\
0 & \textrm{if $d > R$,}
\end{cases}
\end{split}
\end{align}
we see from \eqref{2.7} that
\begin{align}\label{4.2}
\Lambda_R(n;\h,j) := \frac{1}{j!} \Sums{d \mid P(n;\h)}{d \le R} \mu(d)(\log R/d)^j = \Sums{n \bmod d}{\in \od{d}} \lambda_R(d;j).
\end{align}

We call $\h$ admissible if $\v < p$ for all $p$, and one can prove that this is equivalent to $\mathfrak{S}(\h) \ne 0$, where 
\[
\mathfrak{S}(\h) := \prod_{p}\br{1 - \frac{\v}{p}}\br{1 - \frac{1}{p}}^{-k}
\]
is the singular series for $\h$.

\begin{Lem}\label{L4.1}
Let $H$ be a real number, let $Q = Q(H)$ be a positive integer satisfying \eqref{2.2} and \eqref{2.3}, and let $\h$ be as in \eqref{2.6}, with $k$ fixed. We have
\begin{align}\label{4.3}
\v = k \quad \textrm{for all} \quad p > H.
\end{align}
For $k \le \log H$, $\h$ is admissible if and only if $(Q,h_1\cdots h_{k}) = 1$. Moreover, as $H \to \infty$, for $(Q,h_1\cdots h_k) = 1$ we have
\begin{align}\label{4.4}
\mathfrak{S}(\h) \sim \br{\frac{Q}{\phi(Q)}}^{k}.
\end{align}
\end{Lem}
\begin{proof}
For primes $p$ that do not divide $Q$, we have
\begin{align*}
\od{p} = \{-h_1Q^{-1},\ldots,-h_kQ^{-1}\} \bmod p,
\end{align*}
and hence $1 \le \v \le \min(k,p)$. For such $p$, we have $\v = k$ if and only if the $-h_iQ^{-1}$ are all distinct modulo $p$, that is if and only if $p \nmid \Delta$, where
\[
\Delta = \Delta(\h) := \prod_{1 \le i < j \le k} \abs{h_i - h_j}.
\]
By \eqref{2.2}, $p > H$ implies $p \nmid Q$, and since $1 \le \abs{h_i - h_j} \le H$ for every $i,j$, $p > H$ also implies $p \nmid \Delta$, and hence $\v = k$. We have established \eqref{4.3}.

If some prime $p$ divides $(Q,h_1\cdots h_k)$, then $P(n;\h) \equiv h_1\cdots h_k \equiv 0 \bmod p$ for every $n \bmod p$, hence $\v = p$, and so $\h$ is not admissible if $(Q,h_1\cdots h_k) \ne 1$. If $(Q,h_1\cdots h_k) = 1$, then $P(n;\h) \equiv h_1\cdots h_k \not\equiv 0 \bmod p$, and hence $\v = 0$, for every $p$ dividing $Q$. For every other $p$ we have $1 \le \v \le \min(k,p)$. Then for $k \le \log H$ and $p \nmid Q$, we have $1 \le \v \le k \le \log H < p$ by \eqref{2.3}, hence $\h$ is admissible.

Now assume $H$ is large enough so that $\log H \ge 2k$, and suppose $(Q,h_1\cdots h_k) = 1$. Then for \eqref{4.4}, since $\v = 0$ if $p \mid Q$, it suffices to show that
\begin{align}\label{4.5}
\mathfrak{S}^{\prime}(\h) := \prod_{p \nmid Q} \br{1 - \frac{\v}{p}}\br{1 - \frac{1}{p}}^{-k} \sim 1
\end{align}
as $H$ tends to infinity. We break $\mathfrak{S}^{\prime}(\h)$ into two products according as $p \mid \Delta$ or $p \nmid \Delta$, and use the fact that $\v = k$ for $p \nmid Q\Delta$: 
\begin{align}\label{4.6}
\begin{split}
\mathfrak{S}^{\prime}(\h) & = \prod_{p \nmid Q} \br{1 - \frac{k}{p}}\br{1 + \frac{k - \v}{p - k}}\br{1 - \frac{1}{p}}^{-k} \\ & = 
\prod_{p \nmid Q} \br{1 - \frac{k}{p}}\br{1 - \frac{1}{p}}^{-k} 
\prods{p \nmid Q}{p \mid \Delta} \br{1 + \frac{k - \v}{p - k}}.
\end{split}
\end{align}
In this product $p - k \ne 0$ because, by \eqref{2.3}, $p \nmid Q$ implies $p > \log H \ge 2k$. For the same reason, the logarithm of the first product of the last line of \eqref{4.6} is
\[
\sum_{p \nmid Q}\Br{\br{-\frac{k}{p} - \frac{k^2}{2p^2} - \cdots} - k\br{-\frac{1}{p} - \frac{1}{2p^2} - \cdots}} \ll
k^2\sum_{p > \log H} \frac{1}{p^2} \ll \frac{k^2}{\log H\log \log H}.
\]
For the second product, note that since $k/\log H \le 1/2$, we have
\[
0 < \frac{k - \v}{p - k} \le \frac{k}{p-k} \le \frac{2k}{p} < 1.
\] 
Hence the logarithm of the second product is
\[
\le \sums{p \mid \Delta}{p > \log H} \log \br{1 + \frac{k - \v}{p - k}} \ll \sums{p \mid \Delta}{p > \log H} \frac{k}{p} \ll 
\frac{k}{\log H} \sum_{p \mid \Delta} 1 \ll 
\frac{k\log \Delta}{\log H \log\log \Delta} \ll 
\frac{k^3}{\log\log H}
\]
by the prime number theorem, because $\Delta \le H^{\binom{k}{2}}$. Exponentiating and letting $H$ tend to infinity yields \eqref{4.5}.
\end{proof}

We now assume all of the hypotheses of Proposition \ref{P2.2}. The proof of \eqref{2.9} is almost identical to the proof of Lemma 1 of \cite{GMPY2006}, the only difference being that primes $p \mid Qp_0$ are excluded from the representation of $F(s_1,s_2;\Omega)$, where
\begin{align*}
F(s_1,s_2;\Omega) & := \Sum_{d_1,d_2} \mu(d_1)\mu(d_1) \frac{\abs{\od{[d_1,d_2]}}}{[d_1,d_2]d_1^{s_1}d_2^{s_2}} \\ & = 
\prod_{p \nmid Qp_0} \br{1 - \frac{\v}{p}\br{\frac{1}{p^{s_1}} + \frac{1}{p^{s_2}} - \frac{1}{p^{s_1+s_2}}}}
\end{align*}
in the region of absolute convergence. Since $\v = k$ for $p > H$ by \eqref{4.3}, we put
\[
G(s_1,s_2; \Omega) := F(s_1,s_2;\Omega)\br{\frac{\zeta(s_1+1)\zeta(s_2+1)}{\zeta(s_1+s_2+1)}}^k.
\]

In the proof of Lemma 1 of \cite{GMPY2006}, $G(0,0;\Omega) = \mathfrak{S}(\h)$, but in our situation, we have
\[
G(0,0;\Omega) = \prod_{p \nmid Qp_0} \br{1 - \frac{\v}{p}} \prod_{p} \br{1 - \frac{1}{p}}^{-k} = 
\mathfrak{S}(\h) \prod_{p \mid p_0} \br{1 - \frac{\v}{p}}^{-1},
\]
because $(Q,p_0) = 1$ and $\v = 0$ if $p \mid Q$. The last product is $\sim 1$ by \eqref{2.1}. Now applying \eqref{4.4}, and proceeding as in the proof of Lemma 1 of \cite{GMPY2006}, \eqref{2.9} is established.

The proof of \eqref{2.10} follows that of Lemma 2 of \cite{GMPY2006} very closely: there is one important difference concerning the error
\[
E^{*}(N,q) := \max_{x \le N} \max_{(a,q) = 1} 
\Bigg\vert \sums{p \le x}{p \equiv a \bmod q} \log p - \frac{x}{\phi(q)} \Bigg\vert.
\]
The usual Bombieri-Vinogradov theorem will not suffice here, but the next lemma, which is Lemma 2 of \cite{GPY2006}, will.
\begin{Lem}\label{L4.2}
Let $Q$ be an integer and $Y,M$ be numbers such that
\begin{align}\label{4.7}
Q^2 \le Y \le M, \quad \exp\br{2\sqrt{\log M}} \le Y.
\end{align}
If there is an exceptional modulus $q_0 \le Y$, suppose $p_0 \nmid Q$ for some $p_0 \mid q_0$; otherwise, let $p_0 = 1$. If
\begin{align}\label{4.8}
R^{*} := M^{1/2}Q^{-3}\exp\br{-\sqrt{\log M}},
\end{align}
then we have, with explicitly calculable positive constants $c_1$ and $c_2$,
\begin{align}\label{4.9}
\sums{D \le R^{*}}{(D,Qp_0) = 1}E^{*}(M,QD) \le c_1\frac{M}{Q}\exp\br{-\frac{c_2\log M}{\log Y}}.
\end{align}
\end{Lem}
By \eqref{2.2} -- \eqref{2.5}, we see that \eqref{4.7} is satisfied with 
\[
Y = \exp\br{2cH/(\log H)^2} = N^{2c\epsilon(1+o(1))/(\log\log N)^2},
\]
and $M = 3QN$. We also have
\[
R^2 = N^{1/2 - 2\epsilon'} \le R^{*} = (3QN)^{1/2}Q^{-3}\exp\br{-\sqrt{\log 3QN}},
\]
for all sufficiently large $N$, and
\[
c_2\log M/\log Y = c_2(1+o(1))\log N/\log Y = c_2(1+o(1))(\log\log N)^2/2c\epsilon.
\]
Letting $c_3 = c_2/12c\epsilon$ and putting this into \eqref{4.9}, we deduce from Lemma \ref{L4.2} that
\begin{align}\label{4.10}
\Sum_{D \le R^2} E^{*}(3QN,QD) \ll N(\log N)^{-5c_3\log\log N}
\end{align}
for all sufficiently large $N$.

Now, abbreviating $\lambda_R(d;k+\ell)$ to $\lambda_d$, by \eqref{4.2} we have
\begin{align}\label{4.11}
\begin{split}
\sum_{N < n \le 2N}\vartheta(Qn + h)\Lambda_R(n;\h,k+\ell)^2 & =
\Sum_{d_1,d_2} \lambda_{d_1}\lambda_{d_2} \sums{N < n \le 2N}{[d_1,d_2] \mid P(n;\h)} \vartheta(Qn+h) \\ & = 
\Sum_{d_1,d_2} \lambda_{d_1}\lambda_{d_2} \sums{m \bmod [d_1,d_2]}{\in \od{[d_1,d_2]}}
\sumss{QN+h < p \le 2QN+h}{p \equiv h \bmod Q}{p \equiv Qm+h \bmod [d_1,d_2]} \log p.
\end{split}
\end{align}
We may assume $(Qm+h,[d_1,d_2]) = (Q,[d_1,d_2]) = 1$ in the last sum, so we define
\[
\os{d} := \od{d} \setminus \{m \bmod d : (Qm+h,d) \ne 1\}.
\]
For $d_1,d_2$ with $(Q,[d_1,d_2]) = 1$ and $m \bmod [d_1,d_2] \in \os{[d_1,d_2]}$, we let $h_m \bmod Q[d_1,d_2]$ be the unique congruence class mod $Q[d_1,d_2]$ satisfying $h_m \equiv h \bmod Q$ and $h_m \equiv Qm+h \bmod [d_1,d_2]$. Thus, the last sum in \eqref{4.11} is equal to
\[
\sums{QN+h < p \le 2QN+h}{p \equiv h_m \bmod Q[d_1,d_2]} \log p = 
\frac{2QN+h}{\phi(Q[d_1,d_2])} -  \frac{QN+h}{\phi(Q[d_1,d_2])} + O\br{E^{*}(3QN,Q[d_1,d_2])},
\]
and \eqref{4.11} becomes
\begin{align}\label{4.12}
\frac{QN}{\phi(Q)}\mathcal{T}^{*} + O(\mathcal{E}^{*}),
\end{align}
with
\[
\mathcal{T}^{*} := \Sum_{d_1,d_2} \frac{\lambda_{d_1}\lambda_{d_2} \vs{[d_1,d_2]}}{\phi([d_1,d_2])}, \quad 
\mathcal{E}^{*} := \Sum_{d_1,d_2} \abs{\lambda_{d_1}\lambda_{d_2}} \vs{[d_1,d_2]} E^{*}(3QN,Q[d_1,d_2]).
\]

Now from the definition \eqref{4.1} it is clear that $\abs{\lambda_d} \le (\log R)^{k+\ell}$. Also, as we saw in the beginning of the proof of Lemma \ref{L4.1}, since $(Q,h_1\cdots h_k) = 1$ we have $\v \le k$ for all $p$, and so $\vs{d} \le \vd{d} \le k^{\omega(d)}$ for squarefree $d$. Thus
\begin{align*}
\mathcal{E}^{*} & \le (\log R)^{2(k+\ell)} \Sum_{D \le R^2} \mu^2(D)k^{\omega(D)} E^{*}(3QN,QD) \sum_{[d_1,d_2] = D} 1 \\ & = 
(\log R)^{2(k+\ell)} \Sum_{D \le R^2} \mu^2(D)(3k)^{\omega(D)} E^{*}(3QN,QD). 
\end{align*}
By the trivial inequality
\[
E^{*}(3QN,QD) \ll \frac{QN\log QN}{QD} \ll \frac{N\log N}{D},
\]
and the Cauchy-Schwarz inequality, we have
\begin{multline*}
\Sum_{D \le R^2} \mu^2(D)(3k)^{\omega(D)} E^{*}(3QN,QD) \\ \ll
\br{N\log N \sum_{D \le R^2} \frac{\mu^2(D)(3k)^{2\omega(D)}}{D}}^{1/2}\br{\Sum_{D \le R^2}E^{*}(3QN,QD)}^{1/2}.
\end{multline*}
For positive integers $\kappa$, we have
\[
\sum_{D \le R^2} \frac{\mu^2(D)\kappa^{\omega(D)}}{D} = 
\sum_{d\cdots d_{\kappa} \le R^2} \frac{\mu^2(d_1)\cdots \mu^2(d_{\kappa})}{d_1\cdots d_{\kappa}} 
\ll (\log R^2)^{\kappa} \ll (\log N)^{\kappa},
\]
so combining and applying \eqref{4.10} yields
\begin{align}\label{4.13}
\mathcal{E}^{*} \ll 
N\frac{(\log N)^{2(k+\ell) + (3k)^2/2 + 1/2}}{(\log N)^{-2c_3\log\log N}} 
\le N(\log N)^{-c_3\log\log N}.
\end{align}

We will now evaluate $\mathcal{T}^{*}$, assuming first that $Qx+h \not\in \h$. Let $\hd = \h \cup \{Qx+h\}$ and observe that for $p \nmid Q$,
\[
\vs{p} = |\O{p}{\hd}| - 1 := |\op{p}| - 1.
\]
As with $\vd{d}$, a Chinese remainder theorem argument shows that $\vs{d}$ defines a multiplicative function of $d$. Thus
\[
\vs{[d_1,d_2]} = \prod_{p \mid [d_1,d_2]} \br{|\op{p}| - 1},
\]
provided $[d_1,d_2]$ is squarefree and $(Q,[d_1,d_2]) = 1$, as is the case for $d_1,d_2$ appearing in the sum defining $\mathcal{T}^{*}$. 

We now proceed as in the proof of Lemma 2 of \cite{GMPY2006}: again, the only modification necessary is to $G(0,0;\Omega^{+})$. First note that
\begin{align*}
\mathfrak{S}(\hd)  = \prod_{p} \br{\frac{p - |\op{p}|}{p}}\br{\frac{p}{p-1}}\br{1 - \frac{1}{p}}^{-k} = 
\prod_{p} \br{1 - \frac{|\op{p}| - 1}{p-1}}\br{1 - \frac{1}{p}}^{-k}.
\end{align*}
By \eqref{4.3}, $|\op{p}| = |\hd| = k+1$ for $p > H$, and if
\[
G(s_1,s_2;\Omega^{+}) := 
\prod_{p \nmid Qp_0} \br{1 - \frac{|\op{p}| - 1}{p-1}\br{\frac{1}{p^{s_1}} + \frac{1}{p^{s_1}} - \frac{1}{p^{s_1+s_2}}}}\cdot
\br{\frac{\zeta(s_1+1)\zeta(s_2+1)}{\zeta(s_1+s_2+1)}}^k,
\]
then
\begin{align*}
G(0,0;\Omega^{+}) & = \prod_{p \nmid Qp_0} \br{1 - \frac{|\op{p}| - 1}{p-1}} \prod_{p} \br{1 - \frac{1}{p}}^{-k} \\ & = 
 \mathfrak{S}(\hd) \prod_{p \mid Q} \br{1 + \frac{1}{p-1}}^{-1} \prod_{p \mid p_0} \br{1 - \frac{|\op{p}| - 1}{p-1}}^{-1} \\ & \sim
 \br{\frac{Q}{\phi(Q)}}^k,
\end{align*}
by Lemma \ref{L4.1} and \eqref{2.1}. Therefore
\begin{align}\label{4.14}
\mathcal{T}^{*} \sim \br{\frac{Q}{\phi(Q)}}^k\binom{2\ell}{\ell}\frac{(\log R)^{k+2\ell}}{(k+2\ell)!}.
\end{align}

We remark that since $(Q,h) = (Q,h_1\cdots h_k) = 1$, $\hd$ is admissible (for all sufficiently large $N$) by Lemma \ref{L4.1}, so we do not have to consider the other case as in the proof of Lemma 2 in \cite{GMPY2006}. Combining \eqref{4.14} with \eqref{4.13} and \eqref{4.12} yields the first case of \eqref{2.10}. For the case $Qx+h \in \h$, we observe that, similarly to (2.2) of \cite{GMPY2006}, we have
\[
\sum_{N < n \le 2N}\vartheta(Qn+h)\Lambda_R(n;\h,k+\ell)^2 = \sum_{N < n \le 2N}\vartheta(Qn+h)\Lambda_R(n;\h \setminus \{Qx+h\},k+\ell)^2,
\]
so the above evaluation applies with the translation $k \mapsto k-1$, $\ell \mapsto \ell+1$ to \eqref{4.14}.

\section{Proof of Proposition 2.3}\label{Section 5}

\subsection{Auxiliary lemmas}\label{Subsection 5.1}

To prove Proposition \ref{P2.3}, we will use the following lemmas.
\begin{Lem}\label{L5.1}
Fix integers $q$ and $a$ with $(q,a) = 1$. There is a constant $c(q,a) > 0$, depending only on $q$ and $a$, such that
\[
\prods{p \le x}{p \equiv a \bmod q}\br{1 - \frac{1}{p}} \sim \frac{c(q,a)}{(\log x)^{1/\phi(q)}}
\]
as $x \to \infty$.
\end{Lem}
\begin{proof}
This follows from the prime number theorem for arithmetic progressions. For a more precise estimate, with the constant $c(q,a)$ given explicitly, see \cite[Theorem 1]{W1974}. 
\end{proof}
\begin{Lem}\label{L5.2}
Let $\mathscr{S}(x)$ denote the set of positive integers which are $\le x$ and composed only of primes $p \equiv 1 \bmod q$. There is a constant $c(q) > 0$, depending only on $q$, such that 
\[
\abs{\mathscr{S}(x)} = \br{c(q) + O\br{\frac{1}{\log x}}}\frac{x}{\log x} (\log x)^{1/\phi(q)}.
\]
\end{Lem}
\begin{proof}
See \cite[Lemma 3]{S2000}, in which the constant $c(q)$ is given explicitly.
\end{proof}

The next lemma concerns $\Psi(x,y)$, the number of positive integers which are $\le x$ and free of prime factors $> y$ ($y$-smooth numbers). The ratio $\Psi(x,y)/x$ depends essentially on $u = \log x/\log y$, and for $u$ in a certain range is approximated by $\rho(u)$, where $\rho(u)$ is the Dickman-de Bruijn $\rho$-function, defined as the continuous solution to
\begin{align}\label{5.1}
\begin{split}
\rho(u) & := 
\begin{cases}
1 & \textrm{$0 \le u \le 1$,} \\
\frac{1}{u}\int_{u-1}^u \rho(t) \, dt & \textrm{$u > 1$.}
\end{cases}
\end{split}
\end{align}
\begin{Lem}\label{L5.3}
The estimate
\begin{align}\label{5.2}
\frac{\Psi(y^u,y)}{y^u} = \rho(u)\br{1 + O\br{\frac{\log (u+2)}{\log y}}}
\end{align}
holds uniformly in the range 
\begin{align}\label{5.3}
y \ge 3, \quad 1 \le u \le \exp\br{(\log y)^{3/5 - \delta}},
\end{align}
where $\delta$ is any fixed positive number. The estimate
\begin{align}\label{5.4}
\rho(u) = \exp\br{-u\log u - u\log\log u + O(u)}
\end{align}
holds for $u > 3$, and
\begin{align}\label{5.5}
\frac{\Psi(y^u,y)}{y^u} = \exp\br{-u\log u - u\log\log u + O(u)}
\end{align}
holds uniformly in the range
\begin{align}\label{5.6}
3 < u \le y^{1 - \delta}.
\end{align}
Finally, as $y \to \infty$, 
\begin{align}\label{5.7}
\frac{\Psi(y,(\log y)^A)}{y} = \frac{1}{y^{1/A + o(1)}}
\end{align}
holds for any fixed number $A > 1$.
\end{Lem}
\begin{proof}
We refer to the survey article of Granville \cite{G2008}. The asymptotic \eqref{5.2} was shown to hold for the range \eqref{5.3} by Hildebrand \cite{H1986}: see \cite[(1.8), (1.10)]{G2008}. Hildebrand \cite{H1986} also established that the less precise estimate
\[
\frac{\Psi(y^u,y)}{y^u} = \rho(u)\exp\br{O_{\delta}\br{u\exp\br{-(\log u)^{3/5-\delta}}}}
\]
holds, for any fixed number $\delta > 0$, in the wider range \eqref{5.6}. (See displayed formulas \cite[(1.11), (1.13)]{G2008}.) That \eqref{5.5} holds in the same range can be deduced from \eqref{5.4}. (The estimate \eqref{5.5} is less precise, but sufficient for our purposes.) For the estimate \eqref{5.7}, see \cite[(1.14)]{G2008}. 

The value of the Dickman-de Bruijn $\rho$-function is discussed in \cite[3.7 -- 3.9]{G2008}, and \eqref{5.4} was proved by de Bruijn in \cite{B1951b}. 
\end{proof}

\begin{Lem}\label{L5.4}
Let $\mathscr{P}$ be a subset of the primes. As $y \to \infty$, the estimate
\begin{align}\label{5.8}
\prods{p \le y}{p \in \mathscr{P}}\br{1 - \frac{1}{p}}\sumss{n > y^u}{p \mid n \Rightarrow p \le y}{p \in \mathscr{P}} \frac{1}{n} \le (1+o(1))e^{-\gamma} \int_u^{\infty} \rho(v) \,dv. 
\end{align}
holds uniformly for $u$ satisfying
\begin{align}\label{5.9}
u \ge 1, \quad u = \exp\br{(\log y)^{3/5 - \delta}}, 
\end{align}
where $\delta$ is any fixed positive number.
\end{Lem}
\begin{proof}
Define
\[
\varrho(x,y;\mathscr{P}) := \prods{p \le y}{p \in \mathscr{P}}\br{1 - \frac{1}{p}} \sumss{n \le x}{p \mid n \Rightarrow p \le y}{p \in \mathscr{P}} \frac{1}{n}.
\]
If $\ell \le y$ is prime, then
\[
\varrho(x,y;\mathscr{P}) = \prods{p \le y}{p \in \mathscr{P} \cup \{\ell\}}\br{1 - \frac{1}{p}} \cdot \br{1 - \frac{1}{\ell}}^{-1} 
\sumss{n \le x}{p \mid n \Rightarrow p \le y}{p \in \mathscr{P}} \frac{1}{n}.
\]
Now
\[
\br{1 - \frac{1}{\ell}}^{-1}\sumss{n \le x}{p \mid n \Rightarrow p \le y}{p \in \mathscr{P}} \frac{1}{n} =
\br{1 + \frac{1}{\ell} + \frac{1}{\ell^2} + \cdots} 
\sumss{n \le x}{p \mid n \Rightarrow p \le y}{p \in \mathscr{P}} \frac{1}{n} \ge 
\sumss{m \le x}{p \mid m \Rightarrow p \le y}{p \in \mathscr{P} \cup \{\ell\}} \frac{1}{m},
\]
because every $m$ appearing in the last sum may be written as $n\ell^{\alpha}$ for some $\alpha \ge 0$ and some $n$ appearing in the second last sum. Hence, 
\[
\varrho(x,y;\mathscr{P}) \ge \varrho(x,y;\mathscr{P} \cup \{\ell\}),
\]
and applying this inequality repeatedly, we obtain
\[
\varrho(x,y;\mathscr{P}) \ge \prod_{p \le y} \br{1 - \frac{1}{p}} \sums{n \le x}{p \mid n \Rightarrow p \le y} \frac{1}{n}.
\]
Subtracting both sides from $\varrho(\infty,y;\mathscr{P}) = 1 = \varrho(\infty,y;\{p \le y\})$, we deduce that
\begin{align}\label{5.10}
\prods{p \le y}{p \in \mathscr{P}} \br{1 - \frac{1}{p}} \sumss{n > x}{p \mid n \Rightarrow p \le y}{p \in \mathscr{P}} \frac{1}{n} \le
\prod_{p \le y} \br{1 - \frac{1}{p}} \sums{n > x}{p \mid n \Rightarrow p \le y} \frac{1}{n}.
\end{align}
By partial summation, 
\begin{align}\label{5.11}
 \sums{n > x}{p \mid n \Rightarrow p \le y} \frac{1}{n} = \int_x^{\infty} \frac{d\Psi(t,y)}{t} = 
- \frac{\Psi(x,y)}{x} + \int_x^{\infty} \frac{\Psi(t,y)}{t^2} \, dt \le \int_x^{\infty} \frac{\Psi(t,y)}{t^2} \, dt.
\end{align}

Now we assume $x = y^u$, with $u$ satisfying \eqref{5.9} and $y$ tending to infinity. We will divide the range of the last integral in \eqref{5.11} into three parts. First of all, fix any $\epsilon \in (0,1)$ and suppose $t \ge \exp(y^{\epsilon})$, that is $y \le (\log t)^{1/\epsilon}$. By \eqref{5.7} we have
\[
\frac{\Psi(t,y)}{t^2} \le \frac{\Psi(t,(\log t)^{1/\epsilon})}{t^2} = \frac{1}{t^{1+\epsilon + o(1)}}
\]
as $t$, and hence as $y$, tends to infinity. Thus, we may suppose $y$ is large enough so that $\Psi(t,y)/t^2 \le 1/t^{1 + \epsilon/2}$, say, and
\begin{align}\label{5.12}
\int_{\exp(y^{\epsilon})}^{\infty} \frac{\Psi(t,y)}{t^2} \, dt \le \int_{\exp(y^{\epsilon})}^{\infty} \frac{dt}{t^{1 + \epsilon/2}} = \frac{2}{\epsilon\exp\br{\epsilon y^{\epsilon}/2}}.
\end{align}

For the range $x \le t \le \exp(y^{\epsilon})$, the substitution $t = y^v$ yields
\begin{align}\label{5.13}
\int_{x}^{\exp(y^{\epsilon})} \frac{\Psi(t,y)}{t^2} \, dt & = \log y \int_{u}^{y^{\epsilon}/\log y} \frac{\Psi(y^v,y)}{y^v} \, dv.
\end{align}
Next, we let $u_1 = 2\exp\br{(\log y)^{3/5 - \delta}}$, and for $u_1 \le v \le y^{\epsilon}$, we use the estimate \eqref{5.5}:
\[
\frac{\Psi(y^v,y)}{y^v} = \exp\br{-v\log v - v\log\log v + O(v)} \le \frac{1}{v^v},
\]
where the last inequality holds for all sufficiently large $v$, hence for all sufficiently large $y$. Thus
\begin{align}\label{5.14}
\int_{u_1}^{y^{\epsilon}/\log y} \frac{\Psi(y^v,y)}{y^v} \, dv \le \int_{u_1}^{\infty} \frac{dv}{v^v} \ll \frac{1}{u_1^{u_1}}
\end{align}
for all sufficiently large $y$.

For $u \le v \le u_1$, we use the estimate \eqref{5.2}:
\begin{align}\label{5.15}
\begin{split}
\int_{u}^{u_1} \frac{\Psi(y^v,y)}{y^v} \, dv & = \int_u^{u_1} \rho(v)\br{1 + O\br{\frac{\log (v+2)}{\log y}}} \, dv \\ & = 
(1+o(1)) \int_u^{\infty} \rho(v) \, dv - (1+o(1))\int_{u_1}^{\infty} \rho(v) \, dv.
\end{split}
\end{align}
By \eqref{5.4} we have, similarly to \eqref{5.14}, the estimate
\begin{align}\label{5.16}
\int_{u_1}^{\infty} \rho(v) \, dv \le \int_{u_1}^{\infty} \frac{dv}{v^v} \ll \frac{1}{u_1^{u_1}}
\end{align}
for all sufficiently large $y$.

Combining \eqref{5.11} -- \eqref{5.16}, we see that
\begin{align}\label{5.17}
\int_x^{\infty} \frac{\Psi(t,y)}{t^2} \, dt = 
(1+o(1))\log y\int_u^{\infty} \rho(v) \, dv + O\br{u_1^{-u_1}\log y}
\end{align}
for all sufficiently large $y$. Now by definition \eqref{5.1},
\[
\int_u^{\infty} \rho(v) \, dv \ge \int_{u}^{u+1} \rho(v) \, dv = (u+1)\rho(u+1),
\]
and by \eqref{5.4}, $u_1^{-u_1} = o((u+1)\rho(u+1))$ as $u_1 \ge 2u$, and $u_1$ tends to infinity with $y$. Therefore, combining \eqref{5.17} with \eqref{5.11} in fact gives
\begin{align}\label{5.18}
\sums{n > y^u}{p \mid n \Rightarrow p \le y} \frac{1}{n} \le (1+o(1))\log y \int_u^{\infty} \rho(v) \, dv
\end{align}
as $y \to \infty$, for $u$ in the range \eqref{5.9}. Finally, combining \eqref{5.18} with \eqref{5.10} and applying Mertens' theorem, we obtain \eqref{5.8}. 
\end{proof}

\subsection{The proof of Proposition 2.3}\label{Subsection 5.2}

We are now ready to define $Q$ explicitly. The construction is modelled on that of Shiu's \cite{S2000}. For the rest of this section we let $q \ge 3$ and $a$ be integers with $(q,a) = 1$. If $a \equiv 1 \bmod q$, let
\[
\mathscr{P}(H) := 
\{p \le \log H : p \equiv 1 \bmod q\} \cup \{p \le H/(\log H)^2 : p \not\equiv 1 \bmod q \},
\]
otherwise let
\begin{align*}
\mathscr{P}(H) := & \,
\{p \le \log H : p \equiv 1 \bmod q\} \cup \{p \le H/(\log H)^2 : p \not\equiv 1,a \bmod q\} \\ & \cup
\{t(H) \le p \le H/(\log H)^2 : p \equiv 1 \bmod q\} \cup
\{p \le H/t(H) : p \equiv a \bmod q\},
\end{align*}
with
\[
t(H) := \exp\br{\frac{\log H\log\log\log H}{2\log\log H}},
\]
and put
\begin{align}\label{5.19}
\tilde{Q}(H) := q\prod_{p \in \mathscr{P}(H)} p, \quad Q = Q(H) := q\prods{p \in \mathscr{P}(H)}{p \ne p_0} p. 
\end{align}
We check that \eqref{2.2} -- \eqref{2.5} are indeed satisfied by $Q$: only \eqref{2.4} is not immediate, but it follows from the prime number theorem. 

Analogously to \eqref{2.12}, we define
\begin{align}\label{5.20}
\begin{split}
\tilde{S}(H) & := \{h \in (0,H] : \textrm{$(\tilde{Q}(H),h) = 1$ and $h \equiv a \bmod q$}\}, \\
\tilde{T}(H) & := \{h \in (0,H] : \textrm{$(\tilde{Q}(H),h) = 1$ and $h \not\equiv a \bmod q$}\}.
\end{split}
\end{align}
Proposition \ref{P2.3} will follow from the next lemma.

\begin{Lem}\label{L5.5}
Let $H$ be a real parameter tending to infinity, and let $\tilde{Q}(H)$ be as in \eqref{5.19}. We have
\begin{align}\label{5.21}
\ab{\tilde{T}(H)} \ll \frac{H}{\log H}.
\end{align}
Moreover, there is a constant $A = A(q)$, depending on $q$ at most, such that for all sufficiently large $X$, there is some $H$ satisfying
\begin{align}\label{5.22}
\frac{X}{(\log X)^A} \le H \le X,
\end{align}
such that
\begin{align}\label{5.23}
\ab{\tilde{S}(H)} \gg_q H\frac{\phi(\tilde{Q}(H))}{\tilde{Q}(H)}.
\end{align}
The implied constant in \eqref{5.21} is absolute, and that in \eqref{5.23} depends on $q$ at most.
\end{Lem}

\begin{proof}[Proof of Proposition \ref{P2.3}]
Let $S(H)$ and $T(H)$ be as in \eqref{2.12}. If $p_0 \ne 1$ then by \eqref{2.1} there are at most $H/p_0 < H/\log H$ multiples of $p_0$ in $T(H)$, so
\[
\abs{T(H)} \ll \frac{H}{\log H}
\]
by \eqref{5.21}. We also have $\abs{S(H)} \ge \ab{\tilde{S}(H)}$. An application of Lemma \ref{L5.1} reveals that
\begin{align*}
\frac{\phi(\tilde{Q}(H))}{\tilde{Q}(H)} = \prod_{p \in \mathscr{P}(H)}\br{1 - \frac{1}{p}} \gg_q  
\begin{cases}
\frac{1}{\log H}\br{\frac{\log H}{\log\log H}}^{1/\phi(q)}     & \textrm{if $a \equiv 1 \bmod q$,} \\
\frac{1}{\log H}\br{\frac{\log t(H)}{\log\log H}}^{1/\phi(q)}  & \textrm{if $a \not\equiv 1 \bmod q$.}
\end{cases}
\end{align*}
Therefore, in either case, combining \eqref{5.21} and \eqref{5.23} gives
\[
\ab{S(H)} - \ab{T(H)} \gg \ab{\tilde{S}(H)} - \ab{\tilde{T}(H)} \gg_q H\frac{\phi(\tilde{Q}(H))}{\tilde{Q}(H)} \gg H\frac{\phi(Q(H))}{Q(H)}.
\]
Proposition \ref{P2.3} now follows from Lemma \ref{L5.5}.
\end{proof}

\begin{proof}[Proof of Lemma \ref{L5.5}]
We assume $a \not\equiv 1 \bmod q$ as the case $a \equiv 1 \bmod q$ is similar and simpler. 

There are $\ll H/\log H$ primes in $\tilde{T}(H)$, so let us count the composites $h \in \tilde{T}(H)$. If $h = pm$ for some prime $p > H/(\log H)^2$, with $m > 1$, then $m < (\log H)^2$ is composed only of primes $> \log H$ and $\equiv 1 \bmod q$, by the construction of $\mathscr{P}(H)$. Thus, $m$ must be prime itself, and $p \le H/\log H$. We partition $(H/(\log H)^2, H/\log H]$ into sub-intervals $I_l = (e^{l-1}H/(\log H)^2,e^lH/(\log H)^2]$, and $(\log H,$ $(\log H)^2]$ into sub-intervals $J_l = (\log H, (\log H)^2/e^l]$, $1 \le l \le \log\log H$, and using the prime number theorem, we deduce that the contribution from elements with a large prime factor is at most
\[
\sum_{1 \le l \le \log\log H} \sums{p \in I_l}{p \not\equiv 1 \bmod q} \sums{p' \in J_l}{p \equiv 1 \bmod q} 1 \ll
\sum_{1 \le l \le \log\log H} \frac{e^lH}{(\log H)^3} \frac{(\log H)^2}{e^l\log\log H} \ll \frac{H}{\log H}.
\]

If $h = pm$ with $p \equiv a \bmod q$, then $p > H/t(H)$, and $m < t(H)$ must be composed only of primes $\equiv 1 \bmod q$, a contradiction as $h \not\equiv a \bmod q$. The only elements left uncounted must be composed only of primes $p \equiv 1 \bmod q$ with $\log H < p < t(H)$. By \eqref{5.5}, the number of such elements is at most 
\[
\Psi(H,t(H)) = H\exp \br{-u\log u - u\log\log u + O(u)},
\]
where
\[
u = \frac{\log H}{\log t(H)} = \frac{2\log\log H}{\log\log\log H}.
\]
Thus 
\[
u\log u + u\log\log u + O(u) \sim u\log u \sim 2\log\log H,
\]
and so 
\[
\Psi(H,t(H)) \ll \frac{H}{\log H}.
\]
Combining these estimates yields \eqref{5.21}.

Now suppose $H$ is in the range \eqref{5.22}. To bound the size of $\tilde{S}(H)$ from below we will first do the same for 
\[
S'(X) := \{h \in (0,X] : \textrm{$(Q'(X),h) = 1$ and $h \equiv a \bmod q$}\},
\]
where
\[
Q'(X) := q\prod_{p \in \mathscr{P}'(X)} p, \quad \mathscr{P}'(X) := \mathscr{P}(X) \setminus \{p \le \log X : p \equiv 1 \bmod q\}.
\]

Now $pm \in S'(X)$ if $X/t(X) < p \equiv a \bmod q$ and $m \in \mathscr{S}(X/p)$. We partition $(X/t(X),X]$ into sub-intervals $I_l = (e^{l-1}X/t(X),e^lX/t(X)]$, $1 \le l \le \log t(X)$, and deduce, using the prime number theorem for arithmetic progressions and Lemma \ref{L5.2}, that 
\begin{align}\label{5.24}
\begin{split}
\abs{S'(X)} & \ge \sum_{1 \le l \le \log t(X)} \sums{p \in I_l}{p \equiv a \bmod q} \sum_{m \in \mathscr{S}(t(X)/e^l)} 1 \\ & \gg_q
\sum_{1 \le l \le \frac{1}{2}\log t(X)} \frac{e^lX}{t(X)\log X}\cdot \frac{t(X)}{e^l\log t(X)}(\log t(X))^{1/\phi(q)} \\ & \gg
\frac{X}{\log X}(\log t(X))^{1/\phi(q)}.
\end{split}
\end{align}
Now, we may write any $h \in S'(X)$ uniquely as $h = dm$, where $d$ is composed only of primes $p \le \log X$ with $p \equiv 1 \bmod q$, and $m \in \tilde{S}(X)$. Thus, by \eqref{5.24}, there is a constant $c_1(q) > 0$, depending on $q$ at most, such that for all sufficiently large $X$,
\begin{align}\label{5.25}
c_1(q)\frac{X}{\log X}(\log t(X))^{1/\phi(q)} & \le \abs{S'(X)} = \sumss{d \le X}{p \mid d \Rightarrow p \le \log X}{p \equiv 1 \bmod q} \sums{m \le X/d}{m \in \tilde{S}(X)} 1 \le \sumss{d \le X}{p \mid d \Rightarrow p \le \log X}{p \equiv 1 \bmod q} \ab{\tilde{S}(X/d)}.
\end{align}
The inequality on the right is not immediate: in fact if $Z \le X$, then $\tilde{S}(X) \cap (0,Z] \subseteq \tilde{S}(Z)$. To see this, first note that as all of the functions used to define $\mathscr{P}(X)$ are monotonically increasing with $X$,
\[
\mathscr{P}(Z) \subseteq \mathscr{P}(X) \cup \{t(Z) \le p \le t(X) : p \equiv 1 \bmod q\}.
\]
Suppose $m \in \tilde{S}(X) \cap (0,Z]$, but $m \not\in \tilde{S}(Z)$. Then $p \in \mathscr{P}(Z)$ for some $p \mid m$, but $p \not\in \mathscr{P}(X)$, so $t(Z) \le p \le t(X)$ and $p \equiv 1 \bmod q$. Since $m \equiv a \not\equiv 1 \bmod q$, there must be some $p' \mid m$ with $p' \not\equiv 1 \bmod q$ and $p' \le m/p \le Z/t(Z) \le X/t(X)$. Then $p' \in \mathscr{P}(X)$, a contradiction.

Suppose for a contradiction that for some constant $c_2(q) > 0$, depending on $q$ at most, we have
\begin{align}\label{5.26}
\ab{\tilde{S}(H)} \le \frac{c_1(q)}{3c_2(q)}\frac{H}{\log X}\br{\frac{\log t(X)}{\log\log X}}^{1/\phi(q)}
\end{align}
for all $H$ in the range \eqref{5.22}. Then 
\begin{align}\label{5.27}
\begin{split}
\sumss{d \le (\log X)^A}{p \mid d \Rightarrow p \le \log X}{p \equiv 1 \bmod q} \ab{\tilde{S}(X/d)} & \le 
\frac{c_1(q)}{3c_2(q)}\frac{X}{\log X}\br{\frac{\log t(X)}{\log\log X}}^{1/\phi(q)}
\sumss{d \le (\log X)^A}{p \mid d \Rightarrow p \le \log X}{p \equiv 1 \bmod q} \frac{1}{d} \\ & \le 
\frac{c_1(q)}{3c_2(q)}\frac{X}{\log X}\br{\frac{\log t(X)}{\log\log X}}^{1/\phi(q)}
\prods{p \le \log X}{p \equiv 1 \bmod q}\br{1 - \frac{1}{p}}^{-1} \\ & \le
\frac{c_1(q)}{3}\frac{X}{\log X}\br{\log t(X)}^{1/\phi(q)},
\end{split}
\end{align}
provided $X$ is sufficiently large, and for a suitable choice of $c_2(q)$ (given by Lemma \ref{L5.1}).

Now, by the fundamental lemma of Brun's sieve, we have
\begin{align}\label{5.28}
\ab{\tilde{S}(X/d)} \ll \frac{X}{d} \prod_{p \in \mathscr{P}(X/d)}\br{1 - \frac{1}{p}}
\end{align}
for any $d$. If $(\log X)^A < d \le \sqrt{X}$, then $\log (X/d) \asymp \log X$, and applying Lemma \ref{L5.1} to the sieve upper bound \eqref{5.28}, we see that 
\begin{align}\label{5.29}
\begin{split}
\sumss{(\log X)^A < d \le \sqrt{X}}{p \mid d \Rightarrow p \le \log X}{p \equiv 1 \bmod q} \ab{\tilde{S}(X/d)} \le
c_3(q) \frac{X}{\log X}\br{\frac{\log t(X)}{\log\log X}}^{1/\phi(q)} 
\sumss{(\log X)^A < d \le \sqrt{X}}{p \mid d \Rightarrow p \le \log X}{p \equiv 1 \bmod q} \frac{1}{d}
\end{split}
\end{align}
for some constant $c_3(q) > 0$. By lemmas \ref{L5.4} and \ref{L5.1} respectively, we have
\begin{align}\label{5.30}
\begin{split}
\sumss{(\log X)^A < d \le \sqrt{X}}{p \mid d \Rightarrow p \le \log X}{p \equiv 1 \bmod q} \frac{1}{d} & \le 
\prods{p \le \log X}{p \equiv 1 \bmod q}\br{1 - \frac{1}{p}}^{-1}(1+o(1))e^{-\gamma} \int_A^{\infty} \rho(v) \, dv  \\ & \le
c_4(q)(\log\log X)^{1/\phi(q)}\int_A^{\infty} \rho(v) \, dv
\end{split}
\end{align}
for some constant $c_4(q) > 0$. Now by \eqref{5.4}, 
\[
\int_A^{\infty} \rho(v) \, dv \to 0 \quad \textrm{as} \quad A \to \infty,
\]
so we may choose $A = A(c_1(q),c_3(q),c_4(q)) = A(q)$ so that 
\[
\int_{A}^{\infty} \rho(v) \, dv \le \frac{c_1(q)}{4c_3(q)c_4(q)}.
\]
For any such $A$, combining \eqref{5.29} and \eqref{5.30} yields
\begin{align}\label{5.31}
\sumss{(\log X)^A < d \le \sqrt{X}}{p \mid d \Rightarrow p \le \log X}{p \equiv 1 \bmod q} \ab{\tilde{S}(X/d)} \le 
\frac{c_1(q)}{4}\frac{X}{\log X}(\log t(X))^{1/\phi(q)}.
\end{align}

Finally, using Rankin's trick, we see that
\begin{align}\label{5.32}
\begin{split}
\sumss{\sqrt{X} < d \le X}{p \mid d \Rightarrow p \le \log X}{p \equiv 1 \bmod q} \ab{\tilde{S}(X/d)} & \le 
\sums{\sqrt{X} < d \le X}{p \mid d \Rightarrow p \le \log X} \frac{X}{d}\br{\frac{d}{\sqrt{X}}}^{1/3}  \le 
X^{5/6} \prod_{p \le \log X} \br{1 - \frac{1}{p^{2/3}}}^{-1} \\ & \le
X^{5/6} \exp\br{\sum_{p \le \log X} \frac{3}{p^{2/3}}} \le
X^{5/6}\exp\br{9(\log X)^{1/3}} \\ & =
X^{5/6+o(1)}
\end{split}
\end{align}
by the prime number theorem.

Combining \eqref{5.25}, \eqref{5.27}, \eqref{5.31}, and \eqref{5.32}, we obtain $c_1(q) \le 2c_1(q)/3$, which is absurd. We conclude that for all sufficiently large $X$, there is some $H$ in the range \eqref{5.22} for which
\[
\ab{\tilde{S}(H)} \gg_q \frac{H}{\log X}\br{\frac{\log t(X)}{\log\log X}}^{1/\phi(q)} \gg
\frac{H}{\log H}\br{\frac{\log t(H)}{\log\log H}}^{1/\phi(q)}.
\]
A final application of Lemma \ref{L5.1} shows that this is $\gg_q H\phi(\tilde{Q}(H))/\tilde{Q}(H)$.
\end{proof}

\section{A lower bound}\label{Section 6}

In this section we will show how to obtain a quantitative version of Theorem \ref{T1.1}. We will use the assumptions and notation of sections \ref{Section 3} -- \ref{Section 5}, and show that 
\begin{align}\label{6.1}
\abs{\{p_{r+1} \le Y : \textrm{$p_{r+1} \equiv p_r \equiv a \bmod q$ and $p_{r+1} - p_r < \epsilon\log p_r$}\}} \ge Y^{1/3(\log\log Y)^A}
\end{align}
for all sufficiently large $Y$. Here $A = A(q)$ is the constant given in Lemma \ref{L5.5}. This lower bound could be improved by a sharpening of the range \eqref{5.22} for $H$. 

We will first prove that the estimate
\begin{align}\label{6.2}
\sum_{N < n \le 2N} \Lambda(n;\h,k+\ell)^4 \ll N(\log N)^{19k+4\ell}
\end{align}
holds, with an absolute implied constant. For by \eqref{4.1} and \eqref{4.2}, 
\begin{align}\label{6.3}
\begin{split}
\sum_{N < n \le 2N} \Lambda(n;\h,k+\ell)^4 & =
\Sum_{d_1,\ldots,d_4} \lambda_{d_1}\cdots \lambda_{d_4} \sums{N < n \le 2N}{[d_1,\ldots,d_4] \mid P(n;\h)} 1 \\ & = 
\Sum_{d_1,\ldots,d_4} \lambda_{d_1}\cdots \lambda_{d_4} \sums{m \bmod [d_1,\ldots,d_4]}{\in \od{[d_1,\ldots,d_4]}} 
\sums{N < n \le 2N}{n \equiv m \bmod [d_1,\ldots,d_4]} 1 \\ & \le 
\sums{d_1,\ldots,d_4}{\textrm{squarefree}} \abs{\lambda_{d_1}\cdots \lambda_{d_4}} \sums{m \bmod [d_1,\ldots,d_4]}{\in \od{[d_1,\ldots,d_4]}} 
\br{\frac{N}{[d_1,\ldots,d_4]} + O(1)} \\ & \ll
N(\log R)^{4(k+\ell)} \sums{d_1,\ldots,d_4 \le R}{\textrm{squarefree}}\frac{\vd{[d_1,\ldots,d_4]}}{[d_1,\ldots,d_4]}.
\end{split}
\end{align}
To see the last inequality, note that $[d_1,...,d_4] \le R^4 = N^{1 - 4\epsilon'} = o(N)$, and so $N/[d_1,...,d_4]+O(1) \ll N/[d_1,...,d_4]$.

As observed in Section \ref{Section 4}, $\vd{d} \le k^{\omega(d)}$ for squarefree $d$, so
\begin{align}\label{6.4}
\begin{split}
\sums{d_1,\ldots,d_4 \le R}{\textrm{squarefree}} \frac{\vd{[d_1,\ldots,d_4]}}{[d_1,\ldots,d_4]} & \le
\sum_{D \le R^4}\frac{\mu^2(D)k^{\omega(D)}}{D} \sums{d_1,\ldots,d_4}{[d_1,\ldots,d_4] = D} 1 \\ & =
\sum_{D \le R^4} \frac{\mu^2(D)(15k)^{\omega(D)}}{D} \le 
\prod_{p \le R^4}\br{1 + \frac{15k}{p}} \\ & \ll (\log R^4)^{15k}.
\end{split}
\end{align}
Since $R^4 < N$, combining \eqref{6.3} and \eqref{6.4} yields \eqref{6.2}.

Now choose $N$ so that \eqref{3.3} holds. If we restrict the outer sum in the definition of $\mathscr{L}$ to those $n$ for which $(Qn,Qn+H]$ contains a prime string $p_{r+1} \equiv p_r \equiv a \bmod q$, we remove no positive terms. Thus, if $\sum^{*}$ denotes this restricted sum, then
\begin{align}\label{6.5}
\begin{split}
& \mathscr{L} \le \\ & \frac{1}{N}\br{\frac{\phi(Q)}{Q}}^k \SumS_{N < n \le 2N} 
\br{\sum_{h \in S} \vartheta(Qn+h) - \sum_{h \in T} \vartheta(Qn+h) - \log 3QN}\Lambda_{R}(n;\h,k+\ell)^2.
\end{split}
\end{align}
For each $n \in (N,2N]$, 
\begin{align}\label{6.6}
\sum_{h \in S} \vartheta(Qn+h) - \sum_{h \in T} \vartheta(Qn+h) - \log 3QN \le H\log 3QN,
\end{align}
and by the Cauchy-Schwartz inequality, 
\begin{align}\label{6.7}
\SumS_{N < n \le 2N} \Lambda_{R}(n;\h,k+\ell)^2 \le 
\br{\SumS_{N < n \le 2N} 1 }^{1/2}\br{\sum_{N < n \le 2N}\Lambda_{R}(n;\h,k+\ell)^4}^{1/2}.
\end{align}
Combining \eqref{6.5} -- \eqref{6.7} yields
\[
\SumS_{N < n \le 2N} 1 \ge N^2(Q/\phi(Q))^{2k}\mathscr{L}^2(H\log 3QN)^{-2}\br{\sum_{N < n \le 2N}\Lambda_{R}(n;\h,k+\ell)^4}^{-1}.
\]
Using $H = \epsilon\log N$, $\log 3QN = (1+o(1))\log N$, and $Q/\phi(Q) \ge 1$, then applying \eqref{3.3} and \eqref{6.2}, we see that the right-hand side is $\gg_{k,q} N/(\log N)^{17k+2}$. Since $k$ depends on $\epsilon$, we may write
\begin{align}\label{6.8}
\SumS_{N < n \le 2N} 1 \gg_{\epsilon,q} \frac{N}{(\log N)^{B(\epsilon)}},
\end{align}
where $B(\epsilon)$ is a constant depending on $\epsilon$. 

Now fix a large number $Y$, and let
\[
X := \epsilon\br{1 + \frac{2c\epsilon}{(\log\log Y)^2}}^{-1}\log Y,
\]
with $c > 0$ fixed. By Lemma \ref{L5.5}, we may choose $H$ in the range
\[
X/(\log X)^A \le H \le X
\]
so that \eqref{3.3}, hence \eqref{6.1}, holds with $N = \exp(H/\epsilon)$. By \eqref{2.4}, 
\[
3Q(H)N \le \exp\br{\frac{H}{\epsilon} + \frac{cH}{(\log H)^2}} \le Y,
\]
because
\begin{align*}
\frac{H}{\epsilon} + \frac{cH}{(\log H)^2} & = \frac{H}{\epsilon}\br{1 + \frac{c\epsilon}{(\log H)^2}} \le
\frac{X}{\epsilon}\br{1+ \frac{2c\epsilon}{(\log\log Y)^2}} = \log Y. 
\end{align*}
Here we have used $\log H = (1+o(1))\log X = (1+o(1))\log\log Y$. Also,
\[
\log N = H/\epsilon \ge X/\epsilon(\log X)^A \ge \log Y/2(\log\log Y)^A.
\]
Therefore, using \eqref{6.8} as a lower bound for the number of prime strings up to $Y$, we deduce \eqref{6.1}. (At best, we may have $H = X$, in which case we could deduce a lower bound of $Y^{1 - c'/(\log\log Y)^2}$, for some constant $c' > 0$.)

\section{Concluding remarks}\label{Section 7}

Proposition \ref{P2.2} is similar to a special case of Propositions 1 and 2 of \cite{GPY2006}, which are used to prove that
\[
\liminf_{r \to \infty} \frac{p'_{r+\nu} - p'_{r}}{\phi(q)\log p'_r} \le e^{-\gamma}(\sqrt{\nu} - 1)^2,
\]
where $p'_j$ denotes the $j$th smallest prime in the arithmetic progression $a \bmod q$, $(q,a) = 1$. By considering $H_{\nu} = (\nu - 1 + \epsilon)\log N$ instead of $H$, $Q = Q(H_{\nu})$ instead of $Q(H)$, and
\begin{align*}
& \mathscr{L}_{\nu} := \\ & \frac{1}{N}\br{\frac{\phi(Q)}{Q}}^k \sum_{N < n \le 2N} 
\br{\sum_{h \in S} \vartheta(Qn+h) - \nu\sum_{h \in T} \vartheta(Qn+h) - \nu\log 3QN}\Lambda_{R}(n;\h,k+\ell)^2
\end{align*}
instead of $\mathscr{L}$, it is possible to prove that the interval $(Qn,Qn + H_{\nu}]$ contains a string of $\nu + 1$ consecutive primes $\equiv a \bmod q$, for some $n \in (N,2N]$ and a sequence $N \to \infty$. It may be feasible to prove a similar result with $H_{\nu} = (e^{-\gamma}(\sqrt{\nu} - 1)^2 + \epsilon)\log N$.

\section{Acknowledgements}\label{Section 8}

I would like to thank Andrew Granville, without whose help and encouragement this work would not have materialized. For many productive discussions, my thanks also to Jorge Jim\'enez Urroz, and my colleagues Farzad Aryan, Mohammad Bardestani, Daniel Fiorilli and Kevin Henriot.

\bibliographystyle{article}

\begin{thebibliography}{9}

\def\MR#1{\href{http://www.ams.org/mathscinet-getitem?mr=#1}{MR#1}}
\def\MRS#1{\href{http://www.ams.org/mathscinet-getitem?mr=#1}{(#1)}}

\bibitem{B1951b}
N. G. de Bruijn, 
\href{http://alexandria.tue.nl/repository/freearticles/597496.pdf}{`The asymptotic behaviour of a function occurring in the theory of primes'},
{\em J. Indian Math. Soc. (N.S.)} {\bf 15} (1951), 25--32.
\MR{0043838} \MRS{13:326f}

\bibitem{D2000}
H. Davenport, {\em Multiplicative number theory}, 3rd edn (Revised and with a preface by H. L. Montgomery;
Springer-Verlag, New York, 2000).
\MR{1790423} \MRS{2001f:11001}

\bibitem{GMPY2006}
D. A. Goldston, Y. Motohashi, J. Pintz, and C. Y. Y{\i}ld{\i}r{\i}m,
\href{http://projecteuclid.org/getRecord?id=euclid.pja/1146576181}{`Small gaps between primes exist'}, 
{\em Proc. Japan Acad. Ser. A Math. Sci.} {\bf 82} (2006), 61--65. 
\MR{2222213} \MRS{2007a:11135}

\bibitem{GPY2005}
D. A. Goldston, J. Pintz, and C. Y. Y{\i}ld{\i}r{\i}m, 
\href{http://arxiv.org/PS_cache/math/pdf/0508/0508185v1.pdf}{`Primes in tuples I'}, 
Preprint, 2005, \url{http://arxiv.org/} \url{abs/math/0508185v1}.

\bibitem{GPY2006}
D. A. Goldston, J. Pintz, and C. Y. Y{\i}ld{\i}r{\i}m, 
\href{http://projecteuclid.org/DPubS?verb=Display&version=1.0&service=UI&handle=euclid.facm/1229442618&page=record}{`Primes in tuples III. On the difference $p_{n+\nu}-p_n$'}, 
{\em Funct. Approx. Comment. Math.} {\bf 35} (2006), 79--89. 
\MR{2271608} \MRS{2008f:11102}

\bibitem{GPY2009}
D. A. Goldston, J. Pintz, and C. Y. Y{\i}ld{\i}r{\i}m, 
\href{http://arxiv.org/PS_cache/math/pdf/0508/0508185v1.pdf}{`Primes in tuples I'}, 
{\em Ann. of Math. (2)} {\bf 170} (2009), 819--862.
\MR{2552109}

\bibitem{G2008}
A. Granville,
\href{http://www.dms.umontreal.ca/~andrew/PDF/msrire.pdf}{`Smooth numbers: computational number theory and beyond'}, {\em Algorithmic number theory: lattices, number fields, curves and cryptography} (eds J. P. Buhler and P. Stevenhagen),
Mathematical Sciences Research Institute Publications {\bf 44} (Cambridge University Press, Cambridge, 2008), 267--323.
\MR{2467549}

\bibitem{H1986}
A. Hildebrand,
\href{http://dx.doi.org/10.1016/0022-314X(86)90013-2}{`On the number of positive integers $\le x$ and free of prime factors $>y$'},
{\em J. Number Theory} {\bf 22} (1986), 289--307.
\MR{831874} \MRS{87d:11066}

\bibitem{S2000}
D. K. L. Shiu, 
\href{http://dx.doi.org/10.1112/S0024610799007863}{`Strings of congruent primes'},
{\em J. London Math. Soc. (2)} {\bf 61} (2000), 359--373.
\MR{1760689} \MRS{2001f:11155}

\bibitem{W1974}
K. S. Williams,
\href{http://www.ams.org/leavingmsn?url=http://dx.doi.org/10.1016/0022-314X(74)90032-8}{`Mertens' theorem for arithmetic progressions'},
{\em J. Number Theory} {\bf 6} (1974), 353--359.
\MR{0364137} \MRS{51:392}

\end{thebibliography}

\singlespace

{\small D\'epartement de math\'ematiques et de statistique \\
Universit\'e de Montr\'eal \\
CP 6128, succ. Centre-ville \\
Montr\'eal, Qu\'ebec H3C 3J7 \\
Canada}

{\small E-mail address: \href{mailto:freiberg@dms.umontreal.ca}{freiberg@dms.umontreal.ca}}

\end{document}